\documentclass[12pt]{elsarticle}

\usepackage[utf8]{inputenc}
\usepackage[mathscr]{eucal}

\usepackage{amssymb}
\usepackage{amsmath}
\usepackage{amsthm}
\usepackage{enumitem}

\usepackage{pgf,tikz}
\usetikzlibrary{arrows,matrix}

\numberwithin{equation}{section}
\theoremstyle{plain}
\newtheorem{thm}{Theorem}[section]
\newtheorem{lemma}[thm]{Lemma}
\newtheorem{prop}[thm]{Proposition}
\newtheorem{cor}[thm]{Corollary}
\newtheorem{claim}[thm]{Claim}

\theoremstyle{definition}
\newtheorem{defn}[thm]{Definition}

\newtheorem{remark}[thm]{Remark}

\usepackage{xspace} 

\newcommand{\cl}[1]{\mbox{\ensuremath{\mathbf{#1}}}\xspace}
\renewcommand{\P}{\cl{P}}
\newcommand{\NP}{\cl{NP}}
\newcommand{\coNP}{\cl{co}\text{-}\cl{NP}}

\DeclareMathOperator{\per}{per}

\definecolor{light-gray}{gray}{0.7}
\definecolor{v}{rgb}{0.28,0,0.72}
\definecolor{e}{rgb}{0,1,0.2}
\definecolor{r}{rgb}{1,0,0}

\makeatletter
\def\ps@pprintTitle{%
  \let\@oddhead\@empty
  \let\@evenhead\@empty
  \def\@oddfoot{\reset@font\hfil\thepage\hfil}
  \let\@evenfoot\@oddfoot
}
\makeatother

\begin{document}

\title{Rotor-routing reachability is easy, chip-firing reachability is hard}
\author[l]{Lilla T\'othm\'er\'esz\fnref{fn1}}
\address[l]{MTA-ELTE Egerv\'ary Research Group, P\'azm\'any P\'eter s\'et\'any 1/C, Budapest, Hungary}
\ead{tmlilla@caesar.elte.hu}
\fntext[fn1]{LT was supported by the National Research, Development and Innovation Office of Hungary -- NKFIH, grant no.\ 132488. LT was partially supported by the Counting in Sparse Graphs Lendület Research Group of Rényi Institute.}

\begin{abstract}
Chip-firing and rotor-routing are two well-studied examples of abelian networks. We study the complexity of their respective reachability problems. 
We show that the rotor-routing reachability problem is decidable in polynomial time, and we give a simple characterization of when a chip-and-rotor configuration is reachable from another one.
For chip-firing, it has been known that the reachability problem is in $\P$ if we have a class of graphs whose period length is polynomial (for example, Eulerian digraphs). Here we show that in the general case, chip-firing reachability is hard in the sense that if the chip-firing reachability problem were in $\P$ for general digraphs, then the polynomial hierarchy would collapse to $\NP$.
We encode graphs by their adjacency matrix, and we encode ribbon structures ``succinctly'', only remembering the number of consecutive parallel edges.
\end{abstract}

\begin{keyword}chip-firing \sep rotor-routing \sep reachability \sep computational complexity
\MSC[2020] 68Q17 \sep 05C50 \sep 05C85
\end{keyword}

\maketitle


\section{Introduction}

Chip-firing and rotor-routing are two well-studied examples of abelian networks. Abelian networks are asynchronous networks of processors that sit in the vertices of a digraph and communicate through the edges. They are called ``abelian'' because the final state of the network does not depend on the order in which different processors process their input data. For an introduction to abelian networks, see \cite{Bond-Levine}.

In this paper, we study the complexity of the reachability problem of chip-firing and rotor-routing. Previously, the chip-firing reachability problem was shown to be in $\coNP$ \cite{chip-reach}, and in the special case of polynomial period length (which includes for example Eulerian digraphs), it was shown to be in $\P$ \cite{chip-reach,tezisem}. Here we show that in general the chip-firing reachability problem is hard: if it were solvable in polynomial time, then the polynomial hierarchy would collapse to $\NP$. (See Theorem \ref{thm:reachability_hard}.) To show this, we use the $\NP$-hardness of the related chip-firing halting problem, which was proved in \cite{coeulerian}.

For rotor-routing, reachability was known to be in $\P$ in the special case when the target configuration is recurrent \cite{alg_rotor}. Here we show that rotor-routing reachability is also decidable in polynomial time in the general case, and we give a combinatorial characterization for the reachability. (See Theorem \ref{thm:reachability_char}.) We note that, similarly to the case of chip-firing on Eulerian digraphs, it remains open whether one can determine the stopping configuration of a bounded rotor-routing game in polynomial time.

\subsection{Preliminaries on graphs}
Throughout this paper, $G$ will denote a directed graph with vertex set $V(G)$ and edge set $E(G)$. We allow multiple edges, but no loops.
We denote by $\deg^+(v)$ the out-degree of vertex $v$, and by $d(u,v)$ the number of edges pointing from $u$ to $v$. We encode graphs by their adjacency matrix. This means that if we give a graph $G$ as the input of an algorithm, then the size of the input is $O(|V(G)|+\log(|E(G)|))$, and $|E(G)|$ might be exponential in the input size.

A digraph is said to be strongly connected if for each $u,v\in V(G)$ there is a directed path leading from $u$ to $v$. Each digraph has a unique decomposition into strongly connected components. A component is called a \emph{sink component}, if there is no edge leaving the component. Note that a digraph always has at least one sink-component. A vertex is called a \emph{sink} if its out-degree is zero. In this case, it is a one-element sink component.

We denote by $\mathbb{Z}^{V(G)}$ the set of integer vectors whose coordinates are indexed by the vertices of $G$. $\mathbb{Z}^{V(G)}_{\geq 0}$ denotes the set of vectors with nonnegative integer coordinates. For a vertex $v$, $\mathbf{1}_v$ denotes the vector where the coordinate of $v$ is 1, and the rest of the coordinates are 0.

Both for chip-firing and for rotor-routing, the \emph{Laplacian matrix} of the graph will play an important role. We denote the Laplacian matrix of the digraph $G$ by $L_G$. This is the matrix with coordinates
	$$(L_G)_{uv} = \left\{\begin{array}{cl} 
	-\deg^+ (v) & \text{if $u=v$,}  \\
	d(v,u)  & \text{if $u\neq v$}.\\
	\end{array} \right.
	$$

$0$ is always an eigenvalue of the Laplacian matrix. A non-negative vector $p \in \mathbb{Z}^{V(G)}_{\geq 0}$ will be called a \emph{period vector} for $G$ if  $L_Gp=0$. A non-zero period vector is called \emph{primitive} if its entries have no non-trivial common divisor. The following is known.

\begin{prop}\cite[3.1 and 4.1]{BL92} \label{prop::period}
For a strongly connected digraph $G$ there exists a unique primitive period vector $p_G$, moreover, its coordinates are positive. 
For a general digraph $G$, if $G_1, \dots, G_k$ are the sink components of $G$ and a vector $z \in \mathbb{Z}^{V(G)}$ satisfies $L_G z = 0$ then $z = \sum_{i=1}^k \lambda_i p_i$, where for $i \in \{1, \dots, k\}$, $\lambda_i\in \mathbb{Z}$ and $p_i$ is the primitive period vector of $G_i$ restricted to $V(G_i)$ and zero 
elsewhere.
\end{prop}

For a strongly connected digraph $G$, let us denote the sum of the coordinates of $p_G$ by $\per(G)$. For a general digraph $G$ let $\per(G)=\sum_{i=1}^\ell \per(G_i)$ where $G_1,\dots, G_\ell$ are the strongly connected components of $G$.  

It is easy to see that for a connected Eulerian digraph, the constant 1 vector is the primitive period vector, hence in this case $\per(G)=|V(G)|$. However, in general $\per(G)$ may be exponentially large (for an example see the class of digraphs constructed in the proof of Theorem 2 in \cite{Pham_rotor}).
Despite this, the period vector can be computed
in polynomial time, as shown below.

\begin{prop}\label{p:period_szamolhato}
The primitive period vectors of the sink components can be computed in polynomial time in the input size.
\end{prop}
\begin{proof}
By Tarjan's algorithm \cite{Tarjan}, the strongly connected components, and hence the sink components can be computed in polynomial time.
By \cite[Theorem 1.4.21]{groetschel-lovasz-shrijver}, we can compute in polynomial time an integer solution $\tilde{p}_i$ for $L_{G_i}p=0$ where $G_1 \dots G_k$ are the sink components. One can then compute the greatest common divisor of the coordinates for each $\tilde{p}_i$ and divide to get $p_i$.
\end{proof}

\section{Chip-firing}

In a chip-firing game we consider a digraph $G$ with a pile of chips on each of its nodes. A position of the game, called a \emph{chip configuration} is described by a vector $x \in \mathbb{Z}^{V(G)}$, where $x(v)$ is interpreted as the number of chips on vertex $v \in V(G)$, which might be negative. 

The basic move of the game is \emph{firing} a vertex. It means that this vertex passes a chip to its neighbors along each outgoing edge, and so its number of chips decreases by its out-degree. In other words, firing a vertex $v$ transforms the chip configuration $x$ to $x + L_G\mathbf{1}_v$.

The firing of a vertex $v \in V$ is \emph{legal} with respect to a chip configuration $x$, if $v$ has a non-negative amount of chips after the firing (i.e.\ $x(v) \geq \deg^+(v)$). A \emph{legal game} is a sequence of configurations in which each configuration is obtained from the previous one by a legal firing. For a legal game, let us call the vector $f \in \mathbb{Z}^{V(G)}_{\geq 0}$, where $f(v)$ equals the number of times $v$ has been fired, the \emph{firing vector} of the game.
A game terminates if no firing is legal with respect to the last configuration.
The most appealing property of the chip-firing game is the following ``abelian'' property.

\begin{thm}\cite[Remark 2.4]{BLS91} \label{t:chip-firing_commutative}
From a given initial chip configuration, either every legal game can be continued indefinitely, or every legal game terminates after finitely many steps. The firing vector of every maximal legal game is the same.
\end{thm}

In this section we will be interested in the complexity of the chip-firing reachability problem.
We say that a chip configuration $x_2$ is reachable from a chip configuration $x_1$ if there is a legal game starting in $x_1$ and ending in $x_2$. We denote this by $x_1\leadsto x_2$.

The reachability problem asks whether for chip configurations $x_1$ and $x_2$ on a digraph $G$ we have $x_1\leadsto x_2$.
In the case if the period length of a graph class is polynomial, the reachability problem is known to be in $\P$.

\begin{thm}\cite[Theorem 2.3.13]{tezisem}. Let $G$ be a digraph, and $x$ and $y$ chip configurations on $G$. There is an algorithm
that decides whether $x \leadsto y$, and has a running time which is a polynomial of the input size and the period length of $G$.
\end{thm}

Here, we show that unless the polynomial hierarchy collapses to $\NP$, there cannot be a polynomial algorithm for the reachability problem for general digraphs.
\begin{thm}\label{thm:reachability_hard}
Unless the polynomial hierarchy collapses to $\NP$, there is no polynomial algorithm that decides the chip-firing reachability problem on strongly connected digraphs.
\end{thm}

To show this, we first show that deciding recurrence is easier than deciding reachability, then we show that deciding recurrence already has the above mentioned complexity.

Let us call a chip configuration $x$ recurrent if starting from $x$, there is a nonempty legal game leading back to $x$. 

\begin{claim}\label{cl:reachability_harder_than_recurrence}
If there were a polynomial algorithm for deciding the reachability problem for strongly connected digraphs, then we could decide in polynomial time whether a given chip configuration $x$ on a strongly connected digraph is recurrent.
\end{claim}

For this, we need a couple of definitions and lemmas.

 \begin{lemma} \cite[Lemma 4.3]{BL92} \label{l:per_vektor_kihagyhato}
   Let $p$ be a period vector of a digraph $G$, and suppose that $\alpha=(v_1, v_2, \dots, v_s)$ is a legal sequence of firings on $G$ from some initial chip configuration. Let $\alpha'$ be the sequence obtained from $\alpha$ by deleting the first $p(v)$ occurrence of each vertex $v$ (if $v$ occurs less than $p(v)$ times in $\alpha$, then we delete all of its occurrences). Then $\alpha'$ is also a legal sequence of firings from the same initial configuration.
 \end{lemma}

For a given vector $b \in \mathbb{Z}^{V(G)}_{\geq 0}$, let us call the following
game \emph{$b$-bounded chip-firing game}: We are only allowed to 
make legal firings, and each vertex $v$ can be fired at most $b(v)$ times during the whole game. The $b$-bounded game also has an ``abelian'' property. 

\begin{lemma} \cite[Lemma 1.4]{BL92} \label{l:korlatos_jatek_moho}
For a given bound $b \in \mathbb{Z}^{V(G)}_{\geq 0}$ and initial chip configuration $x$, each maximal 
$b$-bounded chip-firing game with initial chip configuration $x$ has the same firing vector.
\end{lemma}

\begin{proof}[Proof of Claim \ref{cl:reachability_harder_than_recurrence}]
Since our digraph $G$ is strongly connected, the primitive period vector $p_G$ is unique. Hence by Lemma \ref{l:per_vektor_kihagyhato}, $x$ is recurrent if and only if there is a legal game with firing vector $p_G$. By Lemma \ref{l:korlatos_jatek_moho}, this is equivalent to the fact that the maximal $p_G$-bounded game started from $x$ has firing vector $p_G$.

Check if there is any vertex $v$ with $x(v)\geq \deg^+(v)$. If not, then $x$ is stable, hence not recurrent. If yes, then choose such a vertex $v$ and fire it. We show that $x$ is recurrent if and only if $x+L_G \mathbf{1}_v \leadsto x$.

Again by Lemmas \ref{l:per_vektor_kihagyhato} and \ref{l:korlatos_jatek_moho}, $x+L_G\mathbf{1}_v \leadsto x$ is equivalent to the fact that the firing vector of the $(p_G-\mathbf{1}_v)$-bounded game from initial configuration $x+L_G\mathbf{1}_v$ has firing vector $p_G-\mathbf{1}_v$. The claim now follows from the abelian property of the $p_G$-bounded game.
\end{proof}

We prove that if deciding whether a chip configuration on a strongly connected digraph is recurrent were in $\P$ then the polynomial hierarchy would collapse to $\NP$. By Claim \ref{cl:reachability_harder_than_recurrence}, this implies Theorem \ref{thm:reachability_hard}.

To prove our statement about the decision of recurrence, we need to examine the chip-firing halting problem. 
The chip-firing halting problem asks whether for a
given digraph $G$ and chip configuration $x$, the game with initial configuration $x$ on the digraph $G$ terminates after finitely many steps. By Theorem \ref{t:chip-firing_commutative}, this indeed depends only on $x$ and $G$. Let us call a chip configuration $x$ on a digraph $G$ \emph{halting}, if the chip-firing game started from $x$ terminates after finitely many steps, and call it non-halting otherwise.
The halting problem is known to be hard:
\begin{thm}\cite[Corollary 3.2]{coeulerian}\label{thm:halting_NP-complete}
The chip-firing halting problem is $\NP$-complete for strongly connected digraphs.
\end{thm}

We show the following.

\begin{prop}\label{prop:coNP}
If there were a polynomial algorithm deciding whether a chip configuration on a strongly connected digraph is recurrent, then the chip-firing halting problem would be in $\coNP$ for strongly connected digraphs.
\end{prop}

Before proving this statement, let us point out why it implies Theorem \ref{thm:reachability_hard}.

\begin{proof}[Proof of Theorem \ref{thm:reachability_hard}]
By Claim \ref{cl:reachability_harder_than_recurrence} and Proposition \ref{prop:coNP}, the existence of a polynomial algorithm for the reachability problem on strongly connected digraphs would imply that the chip-firing halting problem would be in $\coNP$. By Theorem \ref{thm:halting_NP-complete}, this means that
an $\NP$-complete problem were in $\coNP$. This would imply $\NP =\coNP$ which in turn implies that the polynomial hierarchy collapses to $\NP$.
\end{proof}

For proving Proposition \ref{prop:coNP}, we need a definition and a lemma.

\begin{defn}[Linear equivalence \cite{BN-Riem-Roch}]
Let $G$ be a strongly connected digraph.
For $x, y \in \mathbb{Z}^{V(G)}$, let $x\sim y$ if there exists $z\in \mathbb{Z}_{\geq 0}^{V(G)}$ such that $y = x + L_G z$.  In this case we say that $x$ and $y$ are linearly equivalent.
\end{defn}
One can easily check that for a strongly connected digraph, linear equivalence is indeed an equivalence relation on $\mathbb{Z}^{V(G)}$. The only nontrivial property is symmetry, which holds because the primitive period vector has strictly positive entries for a strongly connected digraph.

\begin{lemma}\label{lemma::termination_equivalence} \cite[Lemma 2.1]{coeulerian}
Let $G$ be a strongly connected digraph and let $x$ and $y$ be chip configurations on $G$. If $x \sim y$, then $x$ is terminating if and only if $y$ is terminating. 
\end{lemma}

\begin{prop}\label{prop:lin_ekv_in_P}\cite[Proposition 8]{chip-reach}
There is a polynomial algorithm that for a given digraph $G$ and chip configurations $x$ and $y$ decides whether there exists an $f\in \mathbb{Z}^{V(G)}_{\geq 0}$ such that $y = x + L_G f$, and if such a vector exists, it computes a reduced such firing vector. 
Specifically, for strongly connected digraphs, linear equivalence is decidable in polynomial time.
\end{prop}

\begin{proof}[Proof of Proposition \ref{prop:coNP}]
Our certificate for the non-halting property of the game with initial configuration $x$ is a recurrent configuration $y$ linearly equivalent to $x$.

We claim that if the game with initial configuration $x$ is non-halting then there exist such a $y$. Indeed, play a legal game starting from $x$. As a vertex can only lose chips when it is fired, and in such a case it cannot go into negative, during the legal game, the number of chips on any vertex $v$ is at least $\min\{x(v),0\}$ at any time.
As the number of chips stays constant, there are only finitely many possible configurations we can see. As we can play indefinitely, we will eventually see a configuration $y$ for the second time. This means we returned to this configuration by a legal game, hence $y$ is recurrent. As we also had $x\leadsto y$, in particular we had $x\sim y$.

Also, the existence of a recurrent $y$ such that $x\sim y$ implies that $x$ is non-halting. Indeed, $y$ is non-halting since we can repeat the legal game transforming $y$ to itself indefinitely. Now Lemma \ref{lemma::termination_equivalence} implies that $x$ is also non-halting.

If recurrence were checkable in polynomial time, then this proof was also checkable in polynomial time, since $x\sim y$ can be checked in polynomial time by Proposition \ref{prop:lin_ekv_in_P}.
\end{proof}

\section{Rotor-routing}

In this section, we show that the rotor-routing reachability problem can be decided in polynomial time.

The rotor-routing game is played on a ribbon digraph.
A \emph{ribbon digraph} is a digraph together with a fixed cyclic ordering of the outgoing edges from $v$ for each vertex $v$.
For an edge $e$ with tail $t$, denote by $e^+$ the outgoing edge following $e$ in the cyclic order at $t$. From this point, we always assume that our digraphs have a ribbon digraph structure. 

Let $G$ be a ribbon digraph.
A \emph{rotor configuration} on $G$ is a function $\varrho$ that assigns to each non-sink vertex $v$ an edge with tail $v$. We call $\varrho(v)$ the \emph{rotor} at $v$.
For a rotor configuration $\varrho$, we call the subgraph with edge set $\{\varrho(v): v\in V(G)\}$ the \emph{rotor subgraph}. See Figure \ref{fig:rr}, where the rotor-edges are shown with bold. We emphasize that we need not have any sink in the graph.

A configuration of the rotor-routing game is a pair $(x,\varrho)$, where $x$ is a chip configuration, and $\varrho$ is a rotor configuration on $G$. We call such pairs \emph{chip-and-rotor configuration}.

Given a chip-and-rotor configuration $(x,\varrho)$, a \emph{routing} at a non-sink vertex $v$ results in the configuration $(x', \varrho')$, where
$\varrho'$ is the rotor configuration with
$$
\varrho'(u) = \left\{\begin{array}{cl} \varrho(u) & \text{if $u\neq v$,}  \\
         \varrho(u)^+ & \text{if $u=v$},
      \end{array} \right.
$$
and $x'=x-\mathbf{1}_v+\mathbf{1}_{v'}$ where $v'$ is the head of $\varrho^+(v)$. See Figure \ref{fig:rr} for an example. Routing at a sink vertex has no effect.

We call the routing at $v$ \emph{legal} (with respect to the configuration $(x,\varrho)$), if $x(v)>0$, i.e.~the routing at $v$ does not create a negative entry at $v$. Note that other vertices might have a negative number of chips. A \emph{legal game} is a sequence of configurations such that each configuration is obtained from the previous one by a legal routing. For a legal game, we call the vector $o\in\mathbb{Z}_{\geq 0}^{V(G)}$ where for each $v\in V(G)$, $o(v)$ is the number of times $v$ has been routed in the game, the \emph{odometer} of the game.

We say that a chip-and-rotor configuration $(x_2,\varrho_2)$ is reachable from a chip-and-rotor configuration $(x_1,\varrho_1)$ if there is a legal game starting in $(x_1,\varrho_1)$ and ending in $(x_2,\varrho_2)$. We denote this by $(x_1,\varrho_1)\leadsto (x_2,\varrho_2)$. The rotor-routing reachability problem asks whether for two given chip-and-rotor configurations $(x_1,\varrho_1)$ and $(x_2,\varrho_2)$ on a digraph $G$, we have $(x_1,\varrho_1)\leadsto (x_2,\varrho_2)$.

As we will be interested in computational questions, let us discuss how we will encode ribbon structures. (Recall that we encode digraphs by their adjacency matrix.) One way to encode the ribbon structure is to simply list the edges around each vertex as in the cyclic order. This would give an encoding whose size is $O(|E(G)|)$. However, if there are many consecutive parallel edges in a cyclic order, 
then we can shorten the description by only writing down how many consecutive instances of the parallel edge follow at this point.
We will use this encoding for ribbon structures.
This way, if say, parallel edges are all consecutive in a ribbon structure then our description has only $O(|V(G)|+\log|E(G)|)$ size and $|E(G)|$ might be exponentially large compared to this.

We prove that rotor-routing reachability is decidable in polynomial time even for this succinct encoding:

\begin{thm}
The rotor-routing reachability problem can be decided in polynomial time, even for multigraphs.
\end{thm}

\begin{remark}
The reason that we are interested in succinct encodings of multigraphs is that for chip-firing, encoding multigraphs by their andacency matrix ($O(|V(G)|+\log|E(G)|$ input size) or in unary encoding ($O(|V(G)|+|E(G)|)$ input size) does make a difference. 
The chip-firing halting problem is currently known to be in $\P$ for Eulerian digraphs in unary encoding, but no polynomial algorithm is known for the $O(|V(G)|+\log|E(G)|)$ input size. (Even though the problem is known to be in $\NP\cap\coNP$.) 
\end{remark}
\begin{figure}
\begin{center}

\begin{tikzpicture}[->,>=stealth',auto,scale=1.2,
                    thick,every node/.style={circle,draw,font=\sffamily\small}]
  \node[label=left:0] (1) at (0, 1) {};
  \node[label=below:0] (2) at (0, -1) {};
  \node[label=below:1] (3) at (-1.2, 0) {};
  \node[label=below:0] (4) at (1.2, 0) {};
  \path[every node/.style={font=\sffamily\small},line width=1.6pt]
    (3) edge [bend right=12] node {} (2)
    (2) edge [bend right=12] node {} (1)
    (1) edge [bend right=12] node {} (3);
  \path[every node/.style={font=\sffamily\small},dashed]
    (1) edge node {} (4)
    (2) edge node {} (4)
    (2) edge [bend right=12] node {} (3)
    (3) edge [bend right=12] node {} (1)
    (1) edge [bend right=12] node {} (2);
\end{tikzpicture}
\hspace{0.2cm}
\begin{tikzpicture}[->,>=stealth',auto,scale=1.2,
                    thick,every node/.style={circle,draw,font=\sffamily\small}]
  \node[label=left:1] (1) at (0, 1) {};
  \node[label=below:0] (2) at (0, -1) {};
  \node[label=below:0] (3) at (-1.2, 0) {};
  \node[label=below:0] (4) at (1.2, 0) {};
  \path[every node/.style={font=\sffamily\small},line width=1.6pt]
    (2) edge [bend right=12] node {} (1)
    (3) edge [bend right=12] node {} (1)
    (1) edge [bend right=12] node {} (3);
  \path[every node/.style={font=\sffamily\small},dashed]
    (1) edge node {} (4)
    (2) edge node {} (4)
    (3) edge [bend right=12] node {} (2)
    (2) edge [bend right=12] node {} (3)
    (1) edge [bend right=12] node {} (2);
\end{tikzpicture}
\hspace{0.2cm}
\begin{tikzpicture}[->,>=stealth',auto,scale=1.2, thick,every node/.style={circle,draw,font=\sffamily\small}]
  \node[label=left:0] (1) at (0, 1) {};
  \node[label=below:1] (2) at (0, -1) {};
  \node[label=below:0] (3) at (-1.2, 0) {};
  \node[label=below:0] (4) at (1.2, 0) {};
  \path[every node/.style={font=\sffamily\small},line width=1.6pt]
    (3) edge [bend right=12] node {} (1)
    (1) edge [bend right=12] node {} (2)
    (2) edge [bend right=12] node {} (1);
  \path[every node/.style={font=\sffamily\small},dashed]
    (1) edge node {} (4)
    (2) edge node {} (4)
    (3) edge [bend right=12] node {} (2)
    (1) edge [bend right=12] node {} (3)
    (2) edge [bend right=12] node {} (3);
\end{tikzpicture}
\end{center}
\caption{Let the ribbon structure be the one coming from the positive orientation of the plane. On the left panel, the leftmost vertex can be legally routed since it has a chip. The routing results in the configuration of the middle panel, where the upper vertex can be routed. Routing that vertex gives the rightmost configuration.}\label{fig:rr}
\end{figure}
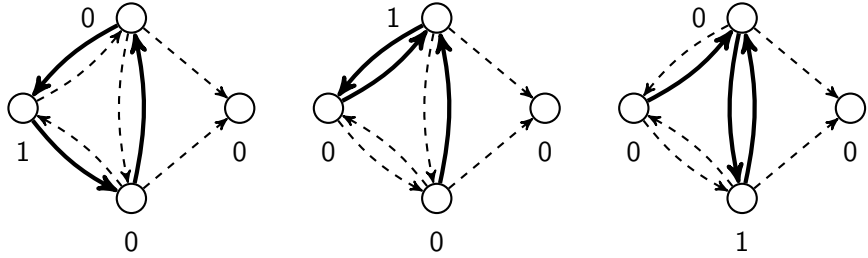

To analyze legal rotor-routing games, it is sometimes convenient to allow non-legal moves. We call a routing an \emph{unconstrained routing} if we perform a routing step so that the routed vertex might not have positive amount of chips.

For some chip-and-rotor configuration $(x,\varrho)$ and vector $r\in \mathbb{Z}^{V(G)}_{\geq 0}$, we denote by $\pi_r(x,\varrho)$ the chip-and-rotor configuration obtained after routing (in an unconstrained way) each vertex $v$ exactly $r(v)$ times from initial configuration $(x,\varrho)$. Note that this is well-defined, and
$\pi_r(x,\varrho)$ is computable in polynomial time since we can compute both the chip configuration and the rotor configuration by a simple calculation.

Similarly to the chip-firing game, it is useful to think about which vectors $r$ have $\pi_r(x,\varrho)=(x,\varrho)$ for some $(x,\varrho)$.
Clearly, in such a case each rotor needs to make some full turns, hence we need to have $r(v)=f(v)\cdot \deg^+(v)$ for each vertex. (If some vertex has $\deg^+(v)=0$, then this formula gives $r(v)=0$, but the routing of these vertices has no effect, so this is reasonable.) For a vector of the form $r(v)=f(v)\cdot \deg^+(v)$, routing $r$ has the same effect on the chip configuration as firing the firing vector $f$. Hence we get back to $(x,\varrho)$ if any only if $r$ is of the form $r(v)=p(v)\cdot \deg^+(v)$ for each $v\in V$ where $p$ is a period vector of $G$. We will call vectors of this form \emph{routing period vector}.

We call a vector \emph{routing reduced}, if it is not coordinatewise greater or equal to any routing period vector. 
Clearly, a vector $r$ is routing reduced, if for the vector $f$ with $f(v)=\lfloor \frac{r(v)}{\deg^+(v)}\rfloor$ for each vertex $v$, $f$ is a reduced firing vector.



We say that $(x_2,\varrho_2)$ is reachable from $(x_1,\varrho_1)$ in the unconstrained sense if there is a vector $r\geq 0$ such that $\pi_r(x_1,\varrho_1)=(x_2,\varrho_2)$. Clearly, in this case $r$ can be chosen to be routing reduced (by subtracting an appropriate routing period vector).

Reachablity in the unconstrained sense is a necessary condition for ``legal'' reachability. 
Fortunately, reachability in the unconstrained sense can be decided in polynomial time: 
\begin{prop}\label{p:rr_lin_ekv_in_P}
There is a polynomial algorithm that for a given digraph $G$ and chip-and-rotor configuations $(x_1,\varrho_1)$ and $(x_2,\varrho_2)$ decides whether there exists a nonnegative integer vector $r$ such that $\pi_r(x_1, \varrho_1)=(x_2,\varrho_2)$. If such a vector exists, a routing reduced $r$ can be computed in polynomial time. 
\end{prop}
\begin{proof}
At each vertex $v$, we need at least as many routings so that the rotor at $v$ turns into the position $\varrho_2(v)$. We can achieve this by routing each vertex $v$ some $r_1(v)<\deg^+(v)$ times. Now we are in a chip-and-rotor configuration $(y,\varrho_2)$ for some chip configuration $y$. We need to determine if there exist a nonnegative vector transforming $(y,\varrho_2)$ to $(x_2,\varrho_2)$. For this, we need a vector $r_2\geq 0$ such that $r_2(v)$ is a multiple of $\deg^+(v)$ for each $v$. Hence the suitable vectors are exactly of the form $r_2(v)=z(v)\cdot\deg^+(v)$ for each $v$ where $z\in\mathbb{Z}_{\geq 0}^{V(G)}$ is a solution to $L_G z=x_2 - y$. By Proposition \ref{prop:lin_ekv_in_P}, the existence of such a $z$ can be decided in polynomial time, and if the answer is yes, a reduced $z$ can also be computed. Now $r=r_1 + r_2$, and as $z$ was reduced and $r_1\leq \deg^+$, $r$ will also be routing reduced.
\end{proof}

Now we can state our condition for the reachability of chip-and-rotor configurations. Note that by Proposition \ref{p:rr_lin_ekv_in_P}, the following condition can be checked in polynomial time.

\begin{thm}\label{thm:reachability_char}
Suppose that $(x,\varrho_x)$ and $(y,\varrho_y)$ are two chip-and-rotor configurations on the digraph $G$. 
Then $(x,\varrho_x)\leadsto (y,\varrho_y)$ if and only if $(y,\varrho_y)$ is reachable from $(x,\varrho_x)$ in the unconstrained sense and for the routing reduced vector $r$ transforming $(x,\varrho_x)$ to $(y,\varrho_y)$, we have \begin{align*}
S_1=\{v\in V \mid y(v) < 0 \text{ and } r(v)>0 \} =\emptyset
\end{align*} and for $T := \{v \in V \mid y(v) = 0 \text{ and }r(v) > 0\}$ we have
$$S_2 := \{v \in T \mid \text{every vertex reachable from $v$ in $\varrho_y$ is contained in $T$} \}=\emptyset. $$
\end{thm}

For proving this theorem, we need some lemmas that are analogous to what we have seen for chip-firing. The following lemma is a special case of \cite[Lemma 4.2]{Abelian_nonhalting}. We include its simple proof for completion.
\begin{lemma}\label{l:rr_per_vektor_kihagyhato}
   Let $p$ be a routing period vector of a digraph $G$, and suppose that $\alpha=(v_1, v_2, \dots, v_s)$ is a legal sequence of routings on $G$ from some initial chip-and-rotor configuration. Let $\alpha'$ be the sequence obtained from $\alpha$ by deleting the first $p(v)$ occurrence of each vertex $v$ (if $v$ occurs less than $p(v)$ times in $\alpha$, then we delete all of its occurrences). Then $\alpha'$ is also a legal sequence of routings from the same initial chip-and-rotor configuration.
\end{lemma}
\begin{proof} The proof is analogous to \cite[Lemma 4.3]{BL92}.
Let $\alpha'=(v_{i_1}, \dots, v_{i_m})$. Suppose by induction that routing $(v_{i_1}, \dots v_{i_{k-1}})$ was legal for some $k$. We show that routing $v_{i_k}$ is also legal.

In the game $\alpha$, one can legally route $v_{i_k}$, hence at that moment, there is a positive amount of chips in it. Compared to $\alpha$, in $\alpha'$ up to this point the vertex $v_{i_k}$ was routed $p(v_{i_k})$ times less, hence it gave out $p(v_{i_k})$ less chips. Up to this point, any in-neighbor $u$ of $v_{i_k}$ routed at most $p(u)$ times less than in $\alpha$. As $p$ is a routing period vector, if each in-neighbor $u$ routed exactly $p(u)$ times less than in $\alpha$, then $v_{i_k}$ would have the same number of chips at its turn as in $\alpha$. If some in-neighbor decreased its number of routings by less than $p(u)$, then $v_{i_k}$ can potentially have more chips at this point than in $\alpha$. Hence $v_{i_k}$ necessarily has the required amount of chips to be able to perform the routing.
\end{proof}

\begin{cor}\label{cor:rr_reach_with_reduced}
If $(x,\varrho_x)\leadsto (y,\varrho_y)$, then there exist a legal game transforming $(x,\varrho_x)$ to $(y,\varrho_y)$ with a routing reduced odometer.
\end{cor}

We have already seen that one can compute in polynomial time whether $(y,\varrho_y)$ is reachable from $(x,\varrho_x)$ in the unconstrained sense, and if yes, give the routing reduced vector transforming $(x,\varrho_x)$ to $(y,\varrho_y)$. Hence for deciding reachability it is now enough to decide if there is a legal game with the given routing reduced vector as odometer. To answer this question, we introduce the bounded game for rotor-routing.

Fix a vector $r\geq 0$. The $r$-bounded rotor-routing game proceeds as follows: If there is a vertex $v$ with positive number of chips such that $v$ has been routed less than $r(v)$ times, then choose one such vertex and route it. If each vertex $v$ either has at most 0 chips or has been routed $r(v)$ times, then the bounded game stops.
This bounded game also has the abelian property: 

\begin{lemma}\label{l:bounded_rr_abelian}
For any initial configuration $(x,\varrho)$ and $r\geq 0$, any maximal $r$-bounded rotor-routing game with initial configuration $(x,\varrho)$ ends in the same chip-and-rotor configuration, and the odometer is the same in each maximal running. 
\end{lemma}
We note that this follows from the general ``abelian theorem'' of \cite{Bond-Levine} as the bounded rotor-routing game is also an abelian network. Still, for completeness we include a direct proof by a variant of an argument of Thorup \cite{thorup}, as this is also very simple.

\begin{proof}
Suppose that there are two maximal bounded games where the odometers are different. Suppose that the odometer of the first running is $o_1$ and the odometer of the second running is $o_2$. By symmetry, we can suppose that there exist a vertex $v_0$ such that $o_1(v_0)<o_2(v_0)$. Play the running with odometer $o_2$ and stop it at the first moment when some vertex $v$ is to be routed for the $o_1(v)+1^{th}$ time. So far, $v$ has transmitted as many chips as altogether in the running with odometer $o_1$. However, all other vertices transmitted at most as many. As the two runnings start from the same initial configuration, the multiset of edge traversals by chips in the stopped second run is a subset of of the multiset of edge traversals by chips in the first run. Hence in particular, $v$ has received at most as many chips in the stopped second run as in the first run. In the second run, $v$ can be routed at this moment, hence it has at least one chip. Thus, $v$ has a chip at the end of the first run, which means $o_1(v)=r(v)$ contradicting the assumption that $r(v)\geq o_2(v)>o_1(v)$.
\end{proof}

We denote by $odom(x,\varrho; r)$ the maximal odometer in the $r$-bounded rotor-routing game, started from $(x,\varrho)$.

\begin{cor}\label{cor:rr_reach_vs_bounded_game}
$(x,\varrho_x)\leadsto (y,\varrho_y)$ is equivalent to the property that  $(y,\varrho_y)$ is reachable from $(x,\varrho_x)$ in the unconstrained sense, and for the routing reduced vector $r$ transforming $(x,\varrho_x)$ to $(y,\varrho_y)$, we have $odom(x,\varrho_x; r) = r$.
\end{cor}

\begin{proof}[Proof of Theorem \ref{thm:reachability_char}]
We first show that the conditions are necessary. Unconstrained reachability is clearly necessary for the reachability, as it means reachability in the weaker sense where no nonnegativity is required for the routings.

We claim that if $(x,\varrho_x)\leadsto (y,\varrho_y)$, then for any $v$ with $r(v)>0$ we need to have $y(v)\geq 0$. Indeed, by Corollary \ref{cor:rr_reach_with_reduced}, in this case there is a legal rotor-routing game from $(x,\varrho_x)$ to $(y,\varrho_y)$ with odometer $r$. If $v$ is routed in a legal game, then after the moment of the first routing, it has a nonnegative amount of chips. Moreover, $v$ can only lose chips by routings, and it cannot go negative by a legal routing. Hence after its first routing, $v$ always has a nonnegative number of chips, thus, $y(v)\geq 0$. This implies $S_1=\emptyset$.

We also claim that if a legal game leads from $(x,\varrho_x)$ to $(y,\varrho_y)$ and has odometer $r$, then in the rotor subgraph $\varrho_y$, from each routed vertex $v$ some vertex $u$ with either $y(u)>0$ or with $r(u)=0$ is reachable. This can be proved by induction for the number of routings. There is nothing to prove if there are no routings. If the statement is true after some routings and we make one more routing, then the additionally routed vertex $v$ has at least one chip before the additional routing. After the routing, the rotor at $v$ points to the vertex $u$ where the chip was transmitted. Either $u$ had at least 0 chips before the routing, in which case now it has a positive amount of chips, or $u$ had a negative number of chips, but then it has not been routed yet. Hence the statement is true for $v$. The rotor-edges of vertices other than $v$ do not change. If $v$ was reachable from some vertex $w$ in the rotor subgraph, then $u$ is reachable from $w$ after the routing. Hence the condition stays true for all other routed vertices. This implies that $S_2=\emptyset$, hence we have proved the necessity of the conditions.

For the sufficiency, it is enough to show that if $(x,\varrho_x)\not\leadsto (y,\varrho_y)$ but $(y,\varrho_y)$ is reachable from $(x,\varrho_x)$ in the unconstrained sense via the primitive routing vector $r$, moreover, $S_1=\emptyset$, then there is a cycle $C$ in $\varrho_y$ with $y(v)=0$ for each $v\in C$, but where each vertex $v\in C$ has $r(v)>0$. In this case all vertices of $C$ are in $S_2$ since in $\varrho_y$ only the vertices of $C$ are reachable from them.

By Corollary \ref{cor:rr_reach_vs_bounded_game}, if $(x,\varrho_x)\not\leadsto (y,\varrho_y)$ but $(y,\varrho_y)=\pi_r(x,\varrho_x)$, then the $r$-bounded rotor-routing game ends so that there there is a nonempty set $Z\subseteq V$ of vertices such that each $v\in Z$ has been routed less than $r(v)$ times, but currently has at most 0 chips. This also implies that $r(v)>0$ for each $v\in Z$. Suppose that this bounded game ends with chip-and-rotor configuration $(z,\varrho_z)$.

Now do the remaining routings in some order, such that each vertex gets routed $r(v)$ times altogether. This will not be a legal game, but nevertheless, at the end, the configuration will be $(y,\varrho_y)$.  Each vertex $v\in Z$ starts from $z(v)\leq 0$ and ends with $y(v)\geq 0$ (since $r(v)> 0$, and we supposed that $S_1=\emptyset$).
As only the vertices of $Z$ are routed in this second phase, in the second phase, vertices of $Z$ can only gain chips from vertices in $Z$. As $z(Z)\leq 0$ and $y(Z)\geq 0$, we conclude that in the second phase vertices only pass chips to vertices in $Z$, and each vertex of $Z$ receives as many chips as it passes away. Specifically, each vertex $v\in Z$ has $z(v)=y(v)=0$.
The final rotor configuration $\varrho_y$ shows for each vertex the edge through which it transmitted its last chip. Hence for each vertex $v\in Z$, $\varrho_y(v)$ is an edge pointing to some vertex in $Z$. This means that each vertex $v\in Z$ has out-degree at least one in $\varrho_y$, and no edge leaves $Z$ in $\varrho_y$. Hence $\varrho_y$ has a cycle $C$ that only contains vertices of $Z$. As $y|_Z\equiv 0$ by our previous argument, this implies that $\varrho_y$ contains a cycle with no chips, but with $r(v)>0$ for $v\in C$. Hence $C\subseteq S_2$. 
\end{proof}

Notice that Theorem \ref{thm:reachability_hard} can be rephrased like this (note that though the vector $r$ was routing reduced in the previous proof, we did not use this property). 
\begin{cor}
Take a chip-and-rotor configuration $(x,\varrho_x)$ and vector $r\geq 0$. Then the $r$-bounded rotor-routing game started from $(x,\varrho_x)$ has maximal odometer $r$ if and only if for $(y, \varrho_y)=\pi_r(x,\varrho_x)$
\begin{align*}
S_1=\{v\in V \mid y(v) < 0 \text{ and } r(v)>0 \} =\emptyset
\end{align*} and for $T := \{v \in V \mid y(v) = 0 \text{ and }r(v) > 0\}$ we have
\begin{align*}
S_2 := \{v \in T \mid \text{every vertex reachable from $v$ in $\varrho_y$ is contained in $T$} \}=\emptyset. 
\end{align*}
\end{cor}

More generally, one could ask what is the complexity of computing the maximal odometer $odom(x,\varrho_x;r)$ for a bounded rotor-routing game. Note that even though we can decide the reachability problem in polynomial time, it is unclear how to compute $odom(x,\varrho_x;r)$.

We note that the computation of the rotor-routing action of Holroyd et al \cite{Holroyd08} is a similar problem, whose complexity is also open.

Finally, we note that this situation is similar to what can be seen for the chip-firing reachability problem for Eulerian digraphs: There also, the reachability problem can be solved in polynomial time, but the computation of the maximal odometer of the bounded game is open \cite{chip-reach}.

\section*{Acknowledgement}
I would like to thank Lionel Levine, Swee Hong Chan, Viktor Kiss and Bálint Hujter for inspiring discussions.

\bibliographystyle{plain}
\bibliography{Bernardi}

\end{document}